\documentclass[12pt]{amsart}
\usepackage{amsmath,
amsfonts,
amsthm,
amssymb,
bbold,
mathrsfs}

\newcommand{\dom}{\operatorname{dom}}
\newcommand{\lex}{<_{\textrm{lex}}}
\newcommand{\comment}[1]{}

\newcommand{\rest}{\upharpoonright}
\newcommand{\Ht}{\operatorname{ht}}
\newcommand{\PFA}{\mathrm{PFA}}

\newcommand{\MA}{\mathrm{MA}}
\newcommand{\Seq}[1]{\langle #1 \rangle}

\newcommand{\ZFC}{\mathrm{ZFC}}

\newcommand{\CH}{\mathrm{CH}}

\newcommand{\Qbb}{\mathbb{Q}}

\newcommand{\range}{\mathrm{range}}

\theoremstyle{plain}
\newtheorem{thm}{Theorem}[section]
\newtheorem{lem}[thm]{Lemma}

\newtheorem{fact}[thm]{Fact}

\newtheorem{question}[thm]{Question}

\theoremstyle{definition}
\newtheorem{defn}[thm]{Definition}

\begin{document}

\author[H. Lamei Ramandi]{Hossein Lamei Ramandi}

\address{Department of Mathematics \\ University of Toronto,
Toronto \\ Canada}

\title[Galvin's Question ]{Galvin's Question on non-$\sigma$-well Ordered Linear Orders}

\subjclass[]{}

\email{{\tt hossein@math.toronto.edu}}

\begin{abstract} 
Assume $\mathcal{C}$ is the class of all linear orders $L$ such that 
$L$ is not a countable union of well ordered sets, 
and every uncountable subset of $L$ contains a copy of $\omega_1$. 
We show it is consistent that $\mathcal{C}$ has minimal elements. 
This answers an old question due to Galvin  in \cite{new_class_otp}.
\end{abstract}
\maketitle
\section{Introduction}
A linear order $L$ is said to be \emph{$\sigma$-well ordered}
if it is a countable union of well ordered subsets.
Galvin asked whether or not every non-$\sigma$-well ordered linear order 
has to contain a real type, Aronszajn type, or $\omega_1^*$.
Baumgartner answered Galvin's question negatively by proving the following theorem.

\begin{thm}[\cite{new_class_otp}]\label{Btypes}
There are non-$\sigma$-well ordered linear orders $L$ such that every uncountable 
suborder of $L$ contains a copy of $\omega_1$.
\end{thm}
Recall that a linear order $L$ is said to be a real type, 
if it is isomorphic to an uncountable set of  real numbers.
An uncountable linear order $L$ is said to be an Aronszajn type, if it does not contain 
any real type or copies of $\omega_1, \omega_1^*$.
Here $\omega_1^*$ is $\omega_1$ with the reverse ordering.

Let $\mathcal{C}$ be the class of all non-$\sigma$-well ordered linear orders $L$
such that every uncountable suborder of $L$ contains a copy of $\omega_1$.
Note that the elements in $\mathcal{C}$ together with real types, Aronszajn types, and $\omega_1^*$
form a basis for the class of non-$\sigma$-well ordered linear orders.
Buamgartner's theorem asserts it is essential to include $\mathcal{C}$ in this basis.

In the final section of \cite{new_class_otp}, 
Baumgartner mentions the following question which is due to Galvin.
\begin{question}[\cite{new_class_otp}, Problem 4]\label{question}
$L \in \mathcal{C}$ is said to be minimal provided that whenever 
$L' \subset L$, $|L'| = |L|$ and $L' \in \mathcal{C}$ then $L$ embeds into $L'$.
Does $\mathcal{C}$ have minimal elements?
\end{question}

Before we answer Question \ref{question}, 
we discuss the motivation behind this question.
The following two deep theorems are about the minimality of non-$\sigma$-well ordered order types.  
\begin{thm}[\cite{club_isomorphic}]
Assume $\MA_{\omega_1}$. Then it is consistent that there is a minimal Aronszajn line.
\end{thm}
\begin{thm}[\cite{reals_isomorphic}]
Assume $\PFA$. Then every two $\aleph_1$-dense\footnote{A linear order is said to be $\aleph_1$-dense if every non-empty interval has size $\aleph_1$.} 
subsets of the reals are isomorphic.
\end{thm} 
\noindent
In particular, these theorems show  it is consistent that real types and Aronszajn types have minimal elements.
It is trivial that $\omega_1^*$ is a minimal non-$\sigma$-well ordered linear order as well. 
So it is natural to ask whether or not $\mathcal{C}$ can have minimal elements.
A consistent negative answer to Question \ref{question} is provided in \cite{no_real_Aronszajn}.
\begin{thm}\label{Ishiu_Moore}
Assume $\PFA^+$. 
Then every minimal non-$\sigma$-scattered linear order is either a real type or an Aronszajn type.
\end{thm}
\noindent
Recall that a linear order $L$ is said to be scattered if it does not contain a copy of $(\Qbb, \leq)$.
$L$ is called $\sigma$-scattered if it is a countable union of scattered suborders. 
Theorem \ref{Ishiu_Moore} provides a consistent negative answer to Question \ref{question} because
of the following fact: 
a linear order $L$ is $\sigma$-well ordered if and only if $L$ is $\sigma$-scattered and  $\omega_1^*$
does not embed into $L$.\footnote{This fact probably exists in classical texts. Since we do not have a 
reference for it, we provide a proof in the next section.}
In this paper we provide a consistent positive answer to Question \ref{question} by proving the following theorem.
\begin{thm}\label{main}
Assume $\mathcal{C}$ is the class of all non-$\sigma$-well ordered linear orders $L$ such that 
every uncountable suborder of $L$ contains a copy of $\omega_1$. Then it is consistent 
with $\ZFC$ that $\mathcal{C}$ has a minimal element of size $\aleph_1$.
\end{thm}
\noindent
This theorem should be compared to the following theorem from \cite{third}.
\begin{thm}\label{third}
It is consistent with $\ZFC$ that there is a minimal non-$\sigma$-scattered linear order $L$,
which does not contain any real type or Aronszajn type.
\end{thm}
\noindent
Theorem \ref{third} does not answer Question \ref{question}.
The reason is that the linear orders which witness Theorem \ref{third} in \cite{third}, 
are dense suborders of the set of all branches of a Kurepa tree $K$. 
Note that such linear orders have to contain copies of $\omega_1^*$.
Moreover, the only way to show that a suborder $L$ of the set of branches of $K$
is not $\sigma$-scattered was to show that $L$ is dense in a Kurepa subtree. 
In particular, it was unclear how to keep the tree $K$ non-$\sigma$-scattered, 
if $K$ had only $\aleph_1$ many branches.
In this paper, aside from eliminating copies of $\omega_1^*$,
we provide a different way of   keeping  $\omega_1$-trees like $K$ 
non-$\sigma$-scattered, in certain forcing extensions.

\section{Preliminaries}
In this section we review some facts and terminology regarding $\omega_1$-trees, linear orders and 
countable support iteration of some type of forcings. 
The material in this section can also be found in \cite{no_real_Aronszajn} and \cite{second}.

Recall that an $\omega_1$-tree is a tree which has height $\omega_1$ and countable levels.
If $T$ is a tree we assume that it does not branch at limit heights.   
More precisely, if $s,t$ are distinct elements in the same level of limit height then they have different
sets of predecessors.
Moreover, we only consider $\omega_1$-trees $T$ that  are ever branching: 
for every element $t \in T$ and $\alpha \in \omega_1$
there are $u,v$ of height more than $\alpha$ which are above $t$ and which  are incomparable. 

Assume $T$ is a tree and $U \subset T$. 
We say that $U$ is \emph{nowhere dense} if for all $t \in T$
there is $s > t$ such that $U$ has no element above $s$. 
If $T$ is a tree and $A$ is a set of ordinals then $T \rest A$ 
is the tree consisting of all $ t \in T$ with $\Ht (t) \in A$.
Assume $T, U$ are trees. The function $f: T \longrightarrow U$ is said to be a \emph{tree embedding} 
if $f$ is one-to-one,
 it is level preserving and $t < s$ if and only if $f(t) < f (s)$. 
Assume $T$ is a tree, then $T_t$ is the collection of all $s \in T$ which are comparable with $t$.  
We call a chain $b \subset T$ a cofinal branch, if it intersects all levels of $T$.
If $b \subset T$ is a branch then $b(\alpha)$ refers to the element $t \in b$ which is of height $\alpha$.  
If $b , b'$ are two different maximal chains then $\Delta(b,b')$ is the smallest ordinal $\alpha$
such that $b(\alpha) \neq b'(\alpha)$. 
The collection of all cofinal branches of $T$ is denoted by $\mathcal{B}(T)$. 
We use the following fact which is easy to check.
\begin{fact}\label{no_omega_1*}
Assume $T$ is a lexicographically ordered $\omega_1$-tree such that 
$(T , \lex )$ has a copy of $\omega_1^*$. 
Then there is a branch $b$ and a sequence of branches $\Seq{b_\xi : \xi \in \omega_1}$
such that: 
\begin{itemize} 
\item for all $\xi \in \omega_1$, $b \lex b_\xi$
\item $\sup \{ \Delta(b, b_\xi) : \xi \in \omega_1 \} = \omega_1$.
\end{itemize}
 \end{fact}
 \begin{defn}\cite{no_real_Aronszajn}
Assume $L$ is a linear order. We use $\hat{L}$ in order to refer to the \emph{completion} of $L$. 
In other words, we add all the Dedekind cuts to $L$ in order  to obtain $\hat{L}$.
For any set $Z$
and $x \in L$ we say $Z$ \emph{captures} 
$x$ if there is $z \in Z \cap \hat{L}$ such that $Z \cap L$ has no element which is strictly in 
between $z$ and $x$. 
\end{defn}
\begin{fact}\cite{no_real_Aronszajn}
Assume $L$ is a linear order, $M \prec H_\theta$ where $\theta$ is a regular large enough cardinal,
$x \in L$ and $M$ captures $x$. Then there is a unique $z \in \hat{L} \cap M$
such that $M\cap L$ has no element strictly in between $z, x$.
In this case we say that $M$ \emph{captures $x$ via $z$}. 
\end{fact}
\begin{defn}\cite{no_real_Aronszajn}
The invariant
$\Omega(L)$ is defined to be the set of all countable $Z \subset \hat{L}$ such that $Z$ captures all elements of $L$. 
We let  $\Gamma(L) = [\hat{L}]^\omega \smallsetminus \Omega(L)$.  
\end{defn}
Assume $T$ is a lexicographically ordered 
$\omega_1$-tree such that for every $t \in T$, there is a cofinal branch $b \subset T$
with $t \in b$.
By $\Omega(T), \Gamma(T)$ we mean $\Omega (\mathcal{B}(T)) , \Gamma (\mathcal{B}(T))$,
where $\mathcal{B}(T)$ is considered with the lexicographic order.
If $M \prec H_\theta$ is countable and $\theta$ is a regular cardinal, we abuse the notation and write
$M \in \Omega(T)$ instead of  $M$ captures all elements of $\mathcal{B}(T)$. 
Similar abuse of notation will be used for $\Gamma$.
The proof of the  following fact is a definition chasing.
\begin{fact}
Assume $T$ is a lexicographically ordered ever branching
$\omega_1$-tree such that for every $t \in T$, there is a cofinal branch $b \subset T$
with $t \in b$. 
Let $\theta $ be a regular cardinal such that $\mathcal{P}(T) \in H_\theta$, 
$M \prec H_\theta$ be countable such that $T \in M$ and $b \in \mathcal{B}(T)$. 
Then $M$ captures $b$ iff there is $c \in \mathcal{B} (T) \cap M$
such that $\Delta(b,c)\geq M \cap \omega_1$.
\end{fact} 
We will use the following lemma in order to characterize $\sigma$-scattered linear orders.
\begin{thm}\cite{no_real_Aronszajn}\label{Omega}
$L$ is $\sigma$-scattered iff $\Gamma(L)$ is not stationary in $[\hat{L}]^\omega$.
\end{thm}

The following lemma will be used in order to determine which linear orders 
are $\sigma$-well ordered. 
Most likely an equivalent of this lemma exists in classical texts, but since we 
did not find any proof for it and for more clarity we include the proof. 
Our proof uses the ideas in the proof of the previous theorem from \cite{no_real_Aronszajn}.
\begin{lem}\label{scattered_wellordered}
Assume $L$ is a linear order  which does not have a copy of $\omega^*_1$. 
Then $L$ is  $\sigma$-well ordered  iff it is $\sigma$-scattered.
\end{lem}
\begin{proof}
Assume $L$ is a linear order of size $\kappa$ which does not have a copy of $\omega^*_1$ 
and which is $\sigma$-scattered.
We use induction on $\kappa$.
In particular, assume that every suborder $L' \subset L$ of size less than $\kappa$ is $\sigma$-well ordered. 
We will show that $L$ is $\sigma$-well ordered.
Let $\theta$ be a regular cardinal such that $\mathcal{P}(L) \in H_\theta$.
Let $\Seq{M_\xi: \xi \in \kappa}$ be a continuous $\in$-chain of  elementary submodels of 
$H_\theta$ such that $L, \Omega(L)$ are in $M_0$.
Moreover assume that, for each $\xi \in \kappa$, $ \xi \subset M_\xi $ and $|M_\xi|= |\xi| + \aleph_0$.
Observe that for all $x \in L$ and $\xi \in \kappa$ there is a unique $z \in \hat{L} \cap M_\xi$
such that $M_\xi$ captures $x$ via $z$.
Moreover, if $M_\xi$ captures $x$ via $z$ then $x \leq z$.
This is because $\omega_1^*$ does not embed into $L$.
Let $\hat{L}_\xi$ be the set of all $z \in \hat{L} \cap M_\xi$ such that for some 
$x \in L$, $M_\xi$ captures $x$ via $z$.
In particular, $|\hat{L}_\xi| \leq |M_\xi| < \kappa$ for each $\xi \in \kappa$.
We note that for each $\xi \in \kappa$, $\hat{L}_\xi$ embeds into $L$.
Then since the size of $\hat{L}_\xi$ is less than $\kappa$, it is $\sigma$-well ordered.

For each $x \in L$ let  $g_x : \kappa \longrightarrow \hat{L}$
such that for all $\xi \in \kappa$, 
$g_x(\xi) \in M_\xi \cap \hat{L}$ and $M_\xi$ captures $x$ via $g_x(\xi)$.
We note that the map $x \mapsto g_x$ is order preserving 
when we consider the lexicoraphic order on all functions from $\kappa$ to $\hat{L}$.
The function $g_x$ is decreasing  because $L$ does not have a copy of $\omega_1^*$.
Also $\range(g_x)$ is finite, because for all $\xi \in \kappa$, $M_\xi$ captures all elements of $L$.

For each $x \in L$ with $|\range(g_x)|= n+1$ we consider the
strictly decreasing finite sequence  $\Seq{z_0(x), z_1(x),..., z_n(x)}$,
such that for each $i \leq n,$ $z_i(x) \in \range(g_x)$. 
Note that if $i < j \leq n,$ $g_x(\xi) = z_i(x), g_x(\eta) = z_j(x)$ then $\xi < \eta$. 
In other words, the cuts $z_i(x)$ appear in the range of $g_x$ in the order we considered in the finite sequence 
associated to $x$.
In particular, $z_n(x) = x$.
For each $i \leq n$, let $\xi_i(x) = \min \{ \xi \in \kappa : g_x (\xi) = z_i(x) \}$.
For each $x \in L$, let $$\sigma(x) = \Seq{z_0(x), \xi_1(x), z_1(x), \xi_2(x), z_2(x),..., \xi_n(x), z_n(x)}.$$

Let $U = \{ \sigma(x) \rest (2n+1) : x \in L \wedge n \in \omega \}$
and $U_m = \{ \sigma(x) \rest (2n+1) : x \in L \wedge n \in m \}$ for each $m \in \omega$.
We consider an order on $U$ as follows. For $\sigma, \tau$ in $U$ we let $\sigma < \tau$ iff
either $\tau$ is an initial segment of $\sigma$ or 
$\sigma \lex \tau$. 
Also let $L_m = \{ x \in L : \sigma(x) \in U_m \}$.
It is easy to see that $x \mapsto \sigma(x)$ is an order preserving map from $L_m$ to $U_m$.
But each  $U_m$ is an iterated sum of $\sigma$-well ordered sets. 
Therefore, $L = \bigcup\limits_{m \in \omega}L_m$ is $\sigma$-well ordered.
\end{proof}
Now we review some definitions and facts about the forcings which we are going to use.
\begin{defn}\cite{second}
Assume $X$ is uncountable and $S \subset [X]^\omega$ is stationary. A poset $P$ is 
said to be \emph{$S$-complete}
if every descending $(M, P)$-generic
sequence $\langle p_n: n\in \omega \rangle$ has a lower bound, for all $M$ with $M \cap X \in S$ and
$M$ suitable for $X,P$. 
\end{defn}
We note that $S$-complete posets preserve the  stationary
subsets of $S$. 
It is also easy to see that if $X,S$ are as above and $P$ is an $S$-complete
forcing, then it preserves $\omega_1$ and adds no new countable sequences of ordinals.
\begin{lem} \cite{second}
Assume $X$ is uncountable and $S\subset[X]^\omega$ is stationary. Then $S$-completeness is 
preserved under countable support iterations.
\end{lem}

\begin{lem} \cite{second}\label{branchcs}
Assume $T$ is an $\omega_1$-tree 
which has no Aronszajn subtree 
in the ground model $\mathbf{V}$. Also assume
 $\Omega(T)$ is stationary  and
$P$ is an $\Omega(T)$-complete forcing.
Then  $T$ has no Aronszajn subtree in $\mathbf{V}^P$.
Moreover, $P$ adds no new branches to $T$.
\end{lem}

\section{A Generic Element of $\mathcal{C}$}

In this section we  introduce the  forcing which adds a generic lexicographically ordered $\omega_1$-tree $T$.
The tree $T$ has no Aronszajn subtrees. 
Moreover, the set of all cofinal branches of $T$, which is denoted by $B$, has no copy of $\omega_1^*$.
In the next section, by iterating two types of posets, we make $B$ a minimal element of $\mathcal{C}$. In particular, 
we will  ensure that $B$ is not $\sigma$-scattered.

\begin{defn}\label{Q}
Fix a set $\Lambda$ of size $\aleph_1$.
The forcing $Q$ is the poset consisting of all conditions $(T_q, b_q, d_q)$ such that the following hold.
\begin{enumerate}
\item  $T_q \subset \Lambda$ is a lexicographically\footnote{
Note that the lexicoraphic order here is independent of
any structure on $\Lambda$ if it exists. 
In other words, this order which we refer to as $\lex$ is determined by the condition $q$.} ordered
 countable tree of height $\alpha_q + 1$ with the property that for all $t \in T_q$ there is 
$s \in (T_q)_{\alpha_q}$ such that $t \leq_{T_q} s$.
\item  $b_q$ is a bijective  map from a countable subset of $\omega_1$ onto $(T_q)_{\alpha_q}$.
\item \label{d}
The map $d_q : \dom(b_q) \longrightarrow \omega_1$ has the property that if 
$b_q(\xi)= t,$ $ b_q(\eta)=s $ and $t \lex s$ then $\Delta(t,s) < d_q(\xi)$.
\end{enumerate}
We let $q \leq p$ if the following hold.
\begin{enumerate}
\item $T_p \subset T_q$ and $(T_p)_{\alpha_p} = (T_q)_{\alpha_p}$. 
\item For all $s,t$ in $T_p$, $s \lex t$ in $T_p$ if and only if $s \lex t$ in $T_q$.
\item For all $s,t$ in $T_p$, $s \leq_{T_p}t$ if and only if $s \leq_{T_q}t$.
\item $\dom(b_p) \subset \dom(b_q)$.
\item For all $\xi \in \dom(b_p)$, $b_p(\xi) \leq_T b_q(\xi)$.
\item  $d_p \subset d_q$. 
\end{enumerate}
\end{defn}

\begin{lem}\label{countably_closed}
Assume $\Seq{q_n : n \in \omega}$ is a decreasing sequence of conditions in $Q$,
$m \leq \omega$
and for each $i \in m$ let
$c_i \subset \bigcup\limits_{n \in \omega} T_{q_n}$ be a cofinal branch. 
Then there is a lower bound $q $ for the sequence $\Seq{q_n : n \in \omega}$
in which every $c_i$ has a maximum with respect to the tree order in $T_q$.
Moreover, for every $t \in (T_q)_{\alpha_q}$ either there is $i \in m$ such that $t$ is above all elements of $c_i$
or there is $\xi \in D= \bigcup\limits_{n \in \omega}\dom(b_{q_n})$ 
such that $t$ is above all elements of $\{b_{q_n}(\xi) : n \in \omega \wedge \xi \in \dom(b_{q_n})\}$. 
In particular, $Q$ is $\sigma$-closed.
\end{lem}
\begin{proof}
For each $n \in \omega$, let $\alpha_n = \alpha_{q_n}$ and $T_n = T_{q_n}$. 
If the set of all $\alpha_n$'s has  a maximum, it means that after some $n$, the sequence $q_n$ is constant.
So without loss of generality assume $\alpha = \sup \{ \alpha_n : n \in \omega \}$ is a limit ordinal above all 
$\alpha_n$'s. Let $T = \bigcup\limits_{n \in \omega} T_n$.
For each $\xi \in D= \bigcup\limits_{n \in \omega}\dom(b_{q_n})$,
let $b_\xi$ be the set of all $t \in T$ such that for some $n \in \omega$, $t \leq b_{q_n}(\xi)$.
Observe that $b_\xi$ is a cofinal branch in $T$.
Since we are going to put an element on top of every $b_\xi$,
from now on, assume that $c_i$'s are different from $b_\xi$.

Now we are ready to define the lower bound $q$.
We let $\alpha_q = \alpha $ and obviously  $(T_q)_{< \alpha} = T$.
We put distinct element $t_\xi$ on top of $b_\xi$ for each $\xi \in D$.
We also put distinct element $s_i$ on top of $c_i$ for each $i \in m$.
Therefore, $(T_q)_\alpha = \{t_\xi : \xi \in D \} \cup \{s_i : i \in m \}$.
Let $E \subset \omega_1 \setminus D$ such that $|E|= m$.
Let  $b_q : D \cup E \longrightarrow (T_q)_\alpha $ be any bijective  function such that
$b_q(\xi) = t_\xi$ for each $\xi \in D$.
For each $\xi \in D$, let $d_q(\xi) = d_{q_n}(\xi)$ where $n \in \omega$ such that $\xi \in \dom(d_{q_n})$.
For each $\eta \in E$ let $d_q(\eta) = \alpha + 1$. 

We need to show that $q$ is a lower bound in $Q$.
We only show Condition \ref{d} of Definition \ref{Q} for $q \in Q$. 
The rest of the conditions and the fact that $q$ is an extension of all $q_n$'s are obvious.
Let $A = \{ t_\xi : \xi \in D \}$, $S = \{ s_i : i \in m \}$, and $u \lex v$ be two distinct elements
in $(T_q)_\alpha$. 
If $u \in S$, then $\Delta(u,v)< \alpha < d_q(\eta)$, where $\eta \in E$ such that  $b_q(\eta) = u$.
If $u,v$ are both in $A$, and $\xi, \xi'$ are in $D$ such that $b_q(\xi) =u, b_q(\xi') =v$,
let $n \in \omega$ such that $\xi , \xi'$ are in $\dom(b_{q_n}).$
Then $\Delta(u,v) = \Delta(b_{q_n}(\xi), b_{q_n}(\xi')) < d_{q_n}(\xi) = d_q(\xi)$.
If $u \in A, v \in S$ and $\xi \in D$ such that $b_q(\xi) = u$, 
fix $n \in \omega$ such that $\xi \in \dom(b_{q_n})$ and $\alpha_{q_n} > \Delta (u,v)$.
Let $u',v'$ be the elements in $T_{\alpha_n}$ which are below $u,v$ respectively.
It is obvious that $b_{q_n}(\xi) = u'$.
Then $\Delta(u,v) = \Delta(u',v') < d_{q_n}(\xi)= d_q(\xi)$.
Therefore $q$ is a condition in $Q$.
\end{proof}

We will use the following terminology and notation regarding the forcing $Q$.
Assume $G$ is a generic filter for $Q$. We let $T= \bigcup_{q \in G}T_q$.
We also let $B=(\mathcal{B}(T), <_{\textrm{lex}})$.
By $b_\xi$ we mean the set of $t \in T$ such that for some $q \in G$, $b_q(\xi)=t$.

\begin{defn}
For every $\xi \in \omega_1$, $d(\xi) =\sup \{ \Delta(b_\xi, b_\eta) : b_\xi <_{\textrm{lex}} b_\eta \}$, and
if $b = b_\xi$ we sometimes use $d(b)$ instead of $d(\xi)$.
\end{defn}
It is worth pointing out that,
by Fact \ref{no_omega_1*}, the role of $d$  is to 
control $\lex$ so that $(B, \lex)$ has no copy of $\omega^*_1$.
The behavior of $d$ plays an essential role from the technical point of view, mostly in the density lemmas 
for the forcings which we introduce and use in the next section.
\begin{lem}\label{countabletoone}
The function $d$ is a countable to one function, i.e. for all $\alpha \in \omega_1$ there are 
countably many $\xi \in \omega_1$ with $d(\xi) = \alpha$.
\end{lem}
\begin{proof}
Assume that the set $A= \{ \xi : d(\xi) = \alpha \}$ is uncountable. Then for each 
pair of distinct ordinals $\xi , \eta$ in $A$, $b_\xi (\alpha+1) \neq b_\eta (\alpha +1)$. 
But this means that  $T$ has an uncountable level which is a contradiction. 
\end{proof}

\begin{lem}\label{rationals_copy}
For every $t_0 \in T$ and $\beta > \Ht(t)$, there is an $\alpha > \beta$ such that 
$(T_\alpha \cap T_{t_0}, \lex)$ contains a copy of the rationals.
\end{lem}
\begin{proof}
We will show that for all $q \in Q$ and $t_0 \in T_q$, the set $\{ p \leq q : (\mathbb{Q},< ) \hookrightarrow 
(\{s \in (T_p)_{\alpha_p} : t \leq_{T_p} s \}, \lex ) \}$ is dense blow $q$.
Fix $r \leq q$ with $\alpha_r > \beta$ and $t \in (T_r)_{\alpha_r} \cap T_{t_0}$.
Let $\xi \in \omega_1$ such that $b_r(\xi) = t$. 
Without loss of generality we can assume that $d_r(\xi) < \alpha_r$.
Fix $X \subset \Lambda \setminus T_r$ an infinite countable set and $u \in X$.
Let  $p < r $ be the condition such that the following hold.
\begin{itemize}
\item $\alpha_p = \alpha_r + 1$, and $\dom(b_p)= \dom(r)\cup E$ where $E$ consists of the first $\omega$
ordinals after $\sup (\dom(r))$.
\item    $T_r \subset T_p$, $(T_r)_{\alpha_r} = (T_p)_{\alpha_r}$ and
for all $s \in (T_r)_{\alpha_r} \setminus \{t \} $ there is a unique 
$s' \in (T_p)_{\alpha_p}$ with $s' > s$.
\item $(T_p)_{\alpha_p}$ consists of the set of all $s'$ as above union with $X$. 
Moreover, for every $x \in X$, $t$ is below $x$, in the tree order.
\item Define $\lex$ on $X$ so that $X$ becomes a 
countable dense linear order without smallest element and with $\max (X) = u$.
\item For every $s \in (T_r)_{\alpha_r} \setminus \{ t\}$ and $s'> s$ in $(T_p)_{\alpha_p}$
let $b_p(s') = b_r(s)$. 
Also $b_p(u) = b_r(t)$. 
Extend $b_p$ on $E$ such that $b_p \rest E$ is a bijection from $E$ to $X \setminus \{ u\}$.
\item The function $d_p$ agrees with $d_r$ on $\dom(b_r)$ and $d_p \rest E$ is constantly $\alpha_p + 1$.
\end{itemize} 
It is easy to see that $p \in Q$ is an extension of $r$ and 
the set $X \setminus \{ u \}$ is a copy of the rationals whose elements are above $t_0$.
\end{proof}

 There is a well known $\sigma$-closed poset which is closely related  to 
our  poset $Q$ and  which generates a Kurepa tree.
Todorcevic showed that the generic Kurepa tree of that poset does not have Aronszajn subtrees. 
The  proof of the following lemma uses the same idea but we include the proof for more clarity.

\begin{lem}\label{all_branches}
Every uncountable downward closed subset of $T$ contains $b_\xi $ for some $\xi \in \omega_1$.
In particular, $\{ b_\xi : \xi \in \omega_1 \}$ is the set of all branches of $T$.
\end{lem}

\begin{proof}
Let $\dot{A}$ be a $Q$-name for an uncountable downward closed subset of $T$ and 
$p \in Q$ forces that $\dot{A}$ contains non of the $b_\xi$'s.
Let $M \prec H_\theta$ be  countable where $\theta$ is a regular large enough regular cardinal such that 
$\dot{A} , p$ are in $M$.
By Lemma \ref{countably_closed} for $m = 0$, there is an $(M,Q)$-generic condition 
$q \leq p$ such that $\alpha_q  = \delta $ where $\delta = M \cap \omega_1$
and for each $t \in (T_q)_\delta$ there is $\xi \in M$ such that $b_q(\xi) =t$.
But then $q$ forces that $\dot{A}$ has no element of height $\delta$ 
which is a contradiction.
\end{proof}
The proof of the following lemma is very similar to the one above.
\begin{lem}
$\Omega(T)$ is stationary.
\end{lem}

We note that if $\CH$ holds then the forcing $Q$ satisfies the $\aleph_2$ chain condition. 
On the other hand, 
if $\kappa > \omega_1$ and 
we consider $\kappa$ many branches for $T$ in the definition of $Q$, then $Q$ collapses $\kappa$
to $\omega_1$.
This is because $Q$ adds a countable to one function from $\kappa$
to $\omega_1$.

\section{Making $B$ Minimal in $\mathcal{C}$}
In this section  we  introduce the forcings which make $B$ a minimal element of $\mathcal{C}$.
The idea is as follows. If  $L \subset B$ is  nowhere dense, we make $L$ $\sigma$-well ordered.
For somewhere dense suborders of $B$ 
we introduce a forcing which adds embedding from $B$ to them and which keeps $B$ 
inside  $\mathcal{C}$.
\begin{defn}
Assume $L \subset B$ is nowhere dense. 
Define $S_L$ to be the poset consisting of all increasing continuous 
sequences $\langle \alpha_i : i \in \beta + 1 \rangle$ in 
$ {\omega_1}^{< \omega_1}$ such that
for all $i \in \beta +1$ and $t \in T_{\alpha_i} \cap( \bigcup L)$
there is $\xi < \alpha_i$ with $t \in b_\xi$. 
We let  $q \leq p$, if $p$  is an initial segment of $q$.
\end{defn}

It is easy to see that for every nowhere dense $L \subset B$, $S_L$ is $\Omega(T)$-complete. 
Therefore, as long as $\Omega(T)$ is stationary, $S_L$ preserves $\omega_1$. 
Moreover, $S_L$ shoots a club in $\Omega (L)$.
So  $S_L$ forces that $L$ is $\sigma$-well ordered. 
This uses  Lemmas \ref{scattered_wellordered}, \ref{Omega}, and the fact that
$L$ has no copy of $\omega_1^*$.

\begin{defn}
Assume $U = T_x$ for some $x \in T$ and $L \subset \mathcal{B}(U)$ is dense in $\mathcal{B}(U)$. Define $E_L$ to be 
the poset consisting of all conditions $q=(f_q, \phi_q)$ such that:
\begin{enumerate}
\item $f_q : T \rest A_q \longrightarrow U\rest A_q$ is a  $\lex$-preserving  tree 
embedding where $A_q$ is a countable and closed subset of $\omega_1$ with $\max(A_q)=\alpha_q$,
\item $\phi_q$ is a countable   partial injection from $\omega_1$ into 
$\{ \xi \in \omega_1 :  b_\xi \in L \}$ such that the map $b_\xi \mapsto b_{\phi_q(\xi)}$ is $\lex$-preserving, 
\item for all $t \in T_{\alpha_q}$ there are at most finitely many $\xi \in \dom(\phi_q) \cup \range(\phi_q)$
with $t \in b_\xi$, 
\item  $f_q , \phi_q$ are consistent, i.e. for all $\xi \in \dom(\phi_q)$, $f_q(b_\xi (\alpha_q)) \in b_{\phi_q(\xi)}$, 
\item for all $\xi \in \dom(\phi_p)$, $d(\xi) \leq d(\phi_p(\xi))$
\end{enumerate} 
We let $q \leq p$ if $A_p$ is an initial segment of $A_q$, $f_p \subset f_q$, and $\phi_p \subset \phi_q$.
\end{defn}

\begin{lem} \label{density2}
For all $\beta \in \omega_1$ the set of all conditions $q \in E_L$ with $\alpha_q > \beta$ is dense in $E_L$.
\end{lem}
\begin{proof}
Fix $p \in E_L$ and
let $D_p = \dom(\phi_p)$ and $R_p = \range(\phi_p)$.
We sometimes abuse the notation and use $D_p, R_p$ in order to refer to the corresponding set of branches,
$\{b_\xi : \xi \in D_p \}$ and $\{b_\xi : \xi \in R_p \}$.
We consider the following partition of $U=T_{\alpha_p} \cap \range(f_p)$.
Let $U_0$ be the set of all $u \in U$ such that if  $u \in b  \in R_p$ then  there is a $c \in B$
with $u \in c$ and $b \lex c$.
Note that if $u \in U$ and there is no $b \in R_p$ with $u \in b$ then $u \in U_0$.
We let $U_1 = U \setminus U_0$.

First we will show that if $u \in U_0$ then there is $\alpha_u \in \omega_1$ and 
$X_u \subset T_{\alpha_u} \cap T_u$ such that:
\begin{enumerate}

\item[a.]
$\alpha_u >\max( \{\Delta(b,c) : b,c \textrm{ are in } A\} \cup \{ \beta\})$, where  
$A$ is the set of all $b \in R_p$ such that $u \in b$,
\item[b.]  $(X_u, \lex)$ is isomorphic to the rationals, and
\item[c.] $\{b(\alpha_u) :  b \in A\} \subset X_u$.
\end{enumerate}
In order to see this,
let $b_m$ be the maximum of $A$ with respect to $\lex$.
This is possible because $A$ is finite.
Let $c \in B$ such that $u \in c$ and $b_m \lex c$. This is possible because we assumed that $u \in U_0$.
Let $t_m$ be the element in $c \setminus b$ which has the lowest height.
By Lemma \ref{rationals_copy}, there is  a copy of the rationals $X_{t_m}$ in some level $\alpha_{t_m} > \Ht(t_m)$
such that for all $x \in X_{t_m}, t_m <_T x$. 
In other words, $u<_T t_m <_T x$ for all $x \in X_{t_m}$ and 
$X_{t_m}$ is isomorphic to the rationals when it is considered with $\lex$.
Moreover, $b_m(\alpha_{t_m}) \lex x$ for all $x \in X_{t_m}.$

Assume $b' \lex b$ are two consecutive elements of $(A, \lex)$.
Let $t_b \in b \setminus b'$ which has the minimum height.
There is $\alpha_{b'b} > \beta$ such that  
$ (T_{\alpha_{b'b}} \cap T_{t_b} , \lex)$ has a copy of the rationals $X_{b'b}$.
Moreover, $X_{b'b}$ can be chosen is such a way that 
for all $x \in X_{b'b}$, $b'(\alpha_{b'b}) \lex x \lex b(\alpha_{b'b})$.
This is because there is no restriction for branching to the left in the tree $T$. 
More precisely, for all $\gamma \in \omega_1$ there is a $c \lex b$ in $B$ such that $\Delta (b,c) > \gamma.$
Similarly, if $a$ is the minimum of $A$ with respect to $\lex$, 
there is $\alpha_a > \beta$ and  $X_a \subset T_{\alpha_a} \cap T_u$ which is isomorphic to the rationals.
Moreover $\alpha_a , X_a$ can be chosen in such a way that if
$x \in X_a$ then $u <_T x$ and for all $b \in A$, $x \lex b(\alpha_a)$.

Now let $\alpha$ be above $\alpha_{t_m}$, $\alpha_a$ and all of $\alpha_{b'b}$'s as above.
Then $T_{\alpha_a} \cap T_u$ contains $X_u$
which is a copy of the rationals and $\{b(\alpha_u) :  b \in A\} \subset X_u$.
 
Note that $U_1$ is the set of all  $u \in U$ such that for some $b_u \in R_p$, $u \in b_u$ and if $u \in c \in B$ then 
$c \leq_{\textrm{lex}} b_u$. 
By the same argument as above we can show 
for all $u \in U_1$  there is $\alpha_u \in \omega_1$ and $X_u \subset T_{\alpha_u} \cap T_u$ such that:
\begin{itemize}
\item[d.]
$\alpha_u >\max( \{\Delta(b,c) : b,c \textrm{ are in } A\} \cup \{ \beta\})$, where  
$A$ is the set of all $b \in R_p$ such that $u \in b$,
\item[e.]  $(X_u \setminus \{b_m(\alpha) \}, \lex)$ is isomorphic to the rationals, where $b_m$ is the maximum of $A$
with respect to $\lex$,
\item[f.] $\{b(\alpha_u) :  b \in A\} \subset X_u$, and 
\item[g.] $\max(X_u , \lex) = b_m(\alpha).$
\end{itemize}

Now we are ready to introduce the extension $q \leq p$. Let $\alpha \in \omega_1$ and $\alpha > \alpha_u$
for all $u \in U$. Let $A_q = A_p \cup \{ \alpha \}$, $\phi_q = \phi_p$.
If $u \in U_0$ then $T_\alpha \cap T_u$ contains $X_u$ such that conditions b,c hold.
If $u \in U_0$ and $f_p(t) =u$, let $f_q \rest (T_\alpha \cap T_t)$ be any $\lex$ preserving function 
which is consistent with $\phi_q$.
If $u \in U_1$ then $T_\alpha \cap T_u$ contains $X_u$ such that conditions e,f,g hold.
In addition, if $u \in U_1, f_q(t) =u,b_m = \max(A)$ and $b$ is mapped to $b_m$ by $\phi_q$,
then $d(b) < d(b_m)$. So $b(\alpha) = \max((T_\alpha \cap T_t), \lex)$.
This means that if $u \in U_1$ and $f_p(t) =u$, we can find $f_q \rest (T_\alpha \cap T_t)$ which is $\lex$ preserving and 
which is consistent with $\phi_q$.
\end{proof}
The  proof of the following  lemma uses Lemma \ref{countabletoone} and the same argument as above.
\begin{lem}\label{density_3}
For all $\xi \in \omega_1$, the set of all $q \in E_L$ with $\xi \in \dom(\phi_q)$ is dense in $E_L$.
\end{lem}
The following lemma shows that $\omega_1$ is preserved by countable support iteration 
of the forcings of the form $E_L$, where $L$ is a somewhere dense subset of $B$.
\begin{lem}
The forcing $E_L$ is $\Omega(T)$-complete.
\end{lem}
\begin{proof}
Assume $x \in T$ and $L$ is dense in $\mathcal{B}(T_x).$
Let $\theta > {2^{\omega_1}}^+$ be a regular cardinal, $M \prec H_\theta$
be countable such that $L, T, x$ are all in $M$ and $M \in \Omega(T)$.
Let $\delta = M \cap \omega_1$, and $\Seq{p_n : n \in \omega}$ be  a decreasing $(M,E_L)$-generic sequence.
We use $A_n , \alpha_n,  f_n, \phi_n$ in order to refer to $A_{p_n}, \alpha_{p_n}, f_{p_n}, \phi_{p_n}$.

We define a lower bound $p= (f_p, \phi_p)$ for $\Seq{p_n : n \in \omega}$ as follows.
Let $A = \bigcup\limits_{n \in \omega} A_n,  A_p = A \cup \delta$.
By Lemma \ref{density2}, $\sup \{ \alpha_n : n \in \omega \} = \delta$, and $A_p$ is closed.
We define $\phi_p = \bigcup\limits_{n \in \omega} \phi_n$.
Note that by elementarity and  Lemma \ref{density_3}, $\bigcup \{ \dom(\phi_n) : n \in \omega \} = \delta$.
So for each $t \in T_\delta$ there is a unique $\xi \in \dom(\phi_p)$ such that $t \in b_\xi$.
Define  $f: T_\delta \longrightarrow T_\delta \cap T_x$ by $f(t)= b_{\phi_p(\xi)}(\delta)$  where 
$\xi $ is the unique $\xi \in \delta$ with $t \in b_\xi$.
Since $\phi_p$ preserves the lexicoraphic order, $f$ does too.
Let $f_p = f \cup \bigcup\limits_{n \in \omega} f_n$.

Since $f_n$ is consistent with $\phi_n$ for each $n$, $f_p$ is a tree embedding.
Obviously $\phi_p$ preserves the lexicographic order. 
Moreover, by elementarity, for each $t \in T_\delta$ there is a unique $\xi \in \delta$ with $t \in b_\xi$.
Since $\phi_p$ is one to one, for each $t \in T_\delta$ there is at most one $\xi \in \range(\phi_p)$ with $t \in b_\xi$.
The rest of the conditions for $p \in E_L$ are obvious.
\end{proof}

Now we are ready to introduce the forcing extension in which $\mathcal{C}$ has a minimal element.
Let's fix some notation.
$P=P_{\omega_2}$ is a countable support iteration 
$\langle P_i , \dot{Q_j} : i \leq \omega_2 , j < \omega_2 \rangle $ over a model of $\CH$  such that 
$Q_0 = Q$, and for all $0<j< \omega_2$, $\dot{Q}_j$ is a  $P_j$-name for either $E_L$ or $S_L$
depending on whether or not $L$ is somewhere dense. 
As usual, the bookkeeping is such that 
if $L \subset B$ is in $\textbf{V}^P$, either $E_L$ or $S_L$ has appeared in some step of the iteration.
This is possible because all of the iterands have size $\aleph_1$ so $P$ satisfies the $\aleph_2$ chain condition.
This can be seen by the work in \cite{second} too.
$T,B$ are the generic objects that are introduced by $Q$, as in the previous section.

Since $E_L, S_L$ are $\Omega(T)$-complete forcings, by Lemma \ref{branchcs}, 
the countable support iteration consisting of the posets $E_L, S_L$
do not add new branches to $T$.
It is worth pointing out that
although $P$ is a $\sigma$-closed forcing, the posets 
$E_L, S_L$ that are involved in the iteration are not even proper. 

\begin{lem} \label{punch}
Assume $G \subset P$ is $\textbf{V}$-generic. Then $\Gamma(B)$ is stationary in $\textbf{V}[G]$. 
\end{lem}
\begin{proof}
Assume $M$ is  a suitable model for $P$ in $\textbf{V}$ 
with $M\cap \omega_1 = \delta$
and $\Seq{p_n : n \in \omega} $ is a descending $(M, P)$-generic 
sequence. 
Let $q_n = p_n \rest 1$ and $R= \bigcup_{n \in \omega} T_{q_n}$. 
Note that if $G \subset P$ is a generic filter over 
$\textbf{V}$ which contains $\Seq{p_n : n \in \omega} $, then  
in $\textbf{V}[G]$ we have
$T_{< \delta} = R$.
We will find an $(M,P)$-generic condition $p$ below $\Seq{p_n : n \in \omega}$
which forces  $M[G]   \in \Gamma(T)$,  in $\textbf{V}[G]$.

Before we work on the details we explain the idea how to find such a condition $p$.
Let $q = p \rest 1$.
Since $\Seq{q_n : n \in \omega}$ is $(M,P_1)$-generic, we have to have that $\dom(b_q) \supset \delta$.
If we allow $\dom(b_q) = \delta$, the advantage is that it is easy to find lower bounds for the 
rest of the sequences $\Seq{p_n (\beta) : n \in \omega}$, by induction on $\beta$.
But then, the resulting lower bound is going to force that $M[G] \notin \Gamma (T)$.
This means that we need to find a lower bound in such a way that $\dom(q) \supsetneq \delta$.
We use a diagonalization argument in $T$ and a Skolem 
closure argument to find such a lower bound for $\Seq{q_n: n \in \omega}$. 
Then we will use induction on $\beta$ to find a lower bound for each $\Seq{p_n (\beta) : n \in \omega}$ with 
$\beta \in M \cap \omega_2$.

Now we return to the proof.
Note that for all $\beta \in M \cap \omega_2$, and $\dot{L}$ a $P_\beta $-name for a nowhere dense subset of $B$
in $M$, $\Seq{p_n : n \in \omega} $ decides $(\bigcup \dot{L}) \cap \dot{T_{< \delta}}$.
More precisely,
there exists $U \subset R$ in $\textbf{V}$ such that
if $G$ is a a $P$-generic filter over $\textbf{V}$ with $\{p_n : n \in \omega \} \subset G$
then $[(\bigcup \dot{L})]_G \cap [\dot{ T_{< \delta}}]_G = U$.
This is because  $[(\bigcup \dot{L})]_G \cap [\dot{ T_{< \delta}}]_G$
 is a countable subset of $R$ and $R \in \textbf{V}$.
Here we use the fact that
countable support iteration of $\Omega(T)$-complete forcings do not add new reals.
Let $\mathcal{U}$ be the set of all countable $U \subset R$ such that for some $\beta \in M \cap \omega_2$ and
 $\dot{L} \in M$
which is a $P_\beta$-name for a  nowhere dense subset of $B$ 
$\Seq{p_n : n \in \omega}$ decides $(\bigcup \dot{L}) \cap \dot{T_{< \delta}}$ to be $U$.

Assume $\beta \in M\cap \omega_2 $ and $\dot{L}$ is a $P_\beta$-name for a somewhere dense 
subset of $B$ and $\dot{Q}_\beta$ is a $P_\beta$-name for the forcing $E_{\dot{L}}$. 
Let's denote the canonical name for
the generic filter of $E_{\dot{L}}$ by $(\dot{f},\dot{ \phi})$. 
Similar to the case of nowhere dense subsets of $B$,
$\Seq{p_n : n \in \omega}$ decides $(\bigcup \dot{L}) \cap \dot{T_{< \delta}}$, $\dot{f} \rest R$ and 
$\dot{\phi} \rest \delta$. 
Moreover, $\dot{f}_G \rest R$ is in $\textbf{V}$, for any $\textbf{V}$-generic filter $G$.

Now, let $\mathcal{F}$ be the set of all finite compositions $g_0 \circ g_1 \circ ... \circ g_n$, such that for all $i \leq n$,
$g_i$ or ${g_i}^{-1}$ is a partial function on $R$ which is of the form $\dot{f}_G \rest R$
where for some $\dot{\phi},$ $(\dot{f}, \dot{\phi})$ is the canonical name for the generic filter added by $E_{\dot{L}}$,
$\dot{L} \in M$ is a name for a somewhere dense subset of $B$, 
and $G$ is a $P$-generic filter over $\textbf{V}$ which contains $\{p_n : n \in \omega \}$.
Also let $\mathcal{W}$ be the collection of all $g[U]$ such that $U \in \mathcal{U}$ and $g \in \mathcal{F}$.
Note that every $W \in \mathcal{W}$ is nowhere dense in $R$. 

Fix $\Seq{W_n : n \in \omega}$ an enumeration 
of $\mathcal{W}$ and $\Seq{\xi_n : n \in \omega}$ an enumeration of $\delta$. Let $\Seq{t_m : m\in \omega}$
be a chain in $R$ such that
\begin{itemize}
\item if $m=2k$ then $W_k$ has no element above $t_m$, and 
\item if $m= 2k+1$ then $t_m$ is not in the downward closure of $\{ b_{q_n}(\xi_k) : n \in \omega \}$.
\end{itemize}

By Lemma \ref{countably_closed}, $\Seq{q_n : n \in \omega}$ has a lower bound $q$
such that  whenever $c$ is a cofinal branch of $R$, then
there is an element on top of $c$ if and only if  one of the following holds.
\begin{itemize}
\item For some $\xi \in \delta$, $\{ b_{q_n}(\xi) : n \in \omega \}$ is cofinal in $c$.
\item For some $g \in \mathcal{F}$,  $\{g(t_m) : m\in \omega\}$ 
is cofinal in $c$.
\end{itemize}
We can choose $q$ in such a way that
if $t$ is on top of $c$, $c$ satisfies the second condition, and $\eta \in \omega_1$ with $b_q(\eta) = t$
then $d_q(\eta) = \delta +1$.
For the rest of the proof, assuming that $p \rest \beta$ is given, we find $p(\beta)$.
If $\beta \notin M \cap \omega_2$ 
we define $p(\beta)$ to be the trivial condition of the corresponding forcing $Q_\beta$.
For each $\beta \in M$, since $p \rest \beta$ is $(M,P_\beta)$-generic, it decides $p_n(\beta)$ for all $n \in \omega$.

Assume  $\dot{Q}_\beta$ is a $P_\beta$-name for some $S_{\dot{L}} \in M$.
We define  $p(\beta) = \{(\delta, \delta) \} \cup \bigcup\limits_{n \in \omega} p_n(\beta)$.
Observe that if $\dot{L}$ is a $P_\beta$-name for a nowhere dense subset of $B$,
 $t \in (T_q)_\delta$, $c_t$ is the set of all elements of 
$T_q$ that are less than  $t$, and $(b_q)^{-1}(t) \geq \delta$
then there is $s \in c_t$ such that $p \rest \beta$ forces that $\bigcup \dot{L}$ has no element above $s$.
This makes $p(\beta)$ a lower bound for $\Seq{p_n(\beta): n \in \omega}$.

Assume  $\dot{Q}_\beta = E_{\dot{L}}$ is a $P_\beta$-name where $\dot{L} \in M$.
By Lemmas \ref{density2}  $\sup(\bigcup\limits_{n \in \omega}A_{p_n(\beta)}) = \delta$.
Moreover, Lemma \ref{density_3} implies that for all $\xi \in \delta$
there is $n \in \omega$ such that $\xi \in \dom(\phi_{p_n(\beta)}).$
We define $p(\beta) = (f,\phi)$ as follows.
Let $\phi = \bigcup_{n \in \omega} \phi_{p_n(\beta)}$, 
$A_{p(\beta)} = \{\delta \}  \cup \bigcup\limits_{n \in \omega}A_{p_n(\beta)} $, $\dom (f) = T_q \rest  A_{p(\beta)} $.
If $\Ht(s) \in A_{p_n(\beta)}$ for some $n \in \omega$, let $f(s)= f_{p_n(\beta)}(s)$. 
If $\Ht(s) = \delta$ and $c$ is a cofinal branch in $R$ whose elements are below $s$, 
let  $f(s)$ be the element on top of the chain $\{ f(v) : v \in c \cap \dom(f) \}$. 
This makes sense, because there is an element on top of $\{ f(v) : v \in c \cap \dom(f) \}$ in $T_q$.
In order to see this, first assume that for some $\xi \in \delta$, $\{ b_{q_n}(\xi) : n \in \omega \}$ is cofinal in $c$.
Let $n \in \omega$ such that $\xi \in \dom(\phi_{p_n(\beta)})$, and $\eta =\phi_{p_n(\beta)}(\xi).$
Then $b_q(\eta)$ is the top element of $\{ f(v) : v \in c \cap \dom(f) \}$.
If for some $g \in \mathcal{F}$,  $\{g(t_m) : m\in \omega \}$ 
is cofinal in $ c$, then  $\{f \circ g(t_m) : m\in \omega \}$ 
 is cofinal in the downward closure of $\{ f(v) : v \in c \cap \dom(f) \}$.
 So $\{ f(v) : v \in c \cap \dom(f) \}$ has a top element as desired.
This finishes defining $f,\phi$. 
It is obvious that $p(\beta)$ is a lower bound for $\Seq{p_n (\beta) : n \in \omega}$.

Now we  show  that if $G \subset P$ is $\textbf{V}$-generic with $p \in G$, then $M[G] \in \Gamma(T)$.
Let $c \subset R$ be the downward closure of $\{ t_m : m \in \omega \}$.
Let $t \in T_\delta$ be the element on top of $c$.
Let $b \in \mathcal{B}(T)$ with $t \in b$.
It is obvious that $c$ is different from the downward closure of $b_q(\xi)$ for each $\xi \in \delta$.
So for all $\xi \in \delta$, $\Delta(b , b_\xi) < \delta$.
It is obvious that $M[G] \cap \omega_1 = \delta$. 
By Lemma \ref{all_branches} and the fact that $\Omega (T)$-complete forcings do not add branches to $T$, 
$\{ b_\xi : \xi \in \omega_1 \}$ is the set of all branches of $T$ in $\textbf{V}[G]$.
Therefore for all $b' \in \mathcal{B}(T) \cap M[G]$, $\Delta(b',b) < \delta$.
This means that $M[G]$ does not capture $b$ in $\textbf{V}[G]$.
\end{proof}

Now we are ready to show that $B$ is a minimal element of $\mathcal{C}$.
Lemma \ref{punch} implies that $B$ is not $\sigma$-scattered and hence it is not $\sigma$-well ordered.
It is obvious that $B$ does not contain any real type.
Since $\Omega(T)$ is stationary in $\textbf{V}[G]$, $B$ does not contain any Aronszajn type either.
Recall that $\omega_1^*$ does not embed into $B$ in $\textbf{V}[G \cap Q]$ and 
$\Omega (T)$-complete forcings do not add new branches to $\omega_1$-trees.
This means that $B$ does not have any copy of $\omega_1^*$ in $\textbf{V}[G]$.
Therefore, every uncountable subset of $B$ in $\textbf{V}[G]$ contains a copy of $\omega_1$.
If $L \subset B$ and $L \in \mathcal{C}$, then $L$ has to be somewhere dense. 
But then the forcing $E_L$ has added  an embedding from $B$ to $L$.
Hence $B$ is minimal in $\mathcal{C}$.

\def\Dbar{\leavevmode\lower.6ex\hbox to 0pt{\hskip-.23ex \accent"16\hss}D}

\end{document}